\documentclass[a4paper,10pt]{article}

\usepackage[margin=30truemm]{geometry}
\usepackage{graphicx}

\usepackage{latexsym}

\usepackage{amsmath,amssymb,amsthm,bbm,amscd}

\usepackage{amsfonts}
\usepackage{bm}
\usepackage{fancybox}
\usepackage{empheq}
\usepackage{mathrsfs}
\usepackage{mathtools}
\usepackage{mleftright}
\usepackage{arydshln}
\usepackage{multirow}

\usepackage{url}

\usepackage[style=alphabetic, backend=biber,maxbibnames=99]{biblatex}
\usepackage{bbm}
\mleftright
\usepackage{tikz}
\theoremstyle{definition}

\allowdisplaybreaks
\renewcommand{\Re}{\operatorname{Re}}

\DeclareMathOperator{\disc}{Disc}
\DeclareMathOperator{\SF}{SF}

\theoremstyle{plain}
\newtheorem{thm}{Theorem}[section]
\newtheorem{prop}[thm]{Proposition}

\usepackage{hyperref}

\title{Secondary terms in the distribution of genus numbers of cubic fields}
\author{Tatsuya Yamada}
\date{\today}

\begin{document}

\maketitle

\begin{abstract}
        We prove the existence of secondary terms of order $X^{5/6}$ in the asymptotic formulas for the average size of the genus number of cubic fields and for the number of cubic fields with a given genus number, establishing improved error estimates. These results refine the estimates obtained by McGown and Tucker. We also provide uniform estimates for the moments of the genus numbers of cubic fields. 
\end{abstract}

\section{Introduction}\label{sec1}
The genus field of a number field $F$ is defined to be the maximal extension $F^\ast$ of $F$ that is unramified at all finite primes and also is a compositum of the form $Fk^\ast$ where $k^\ast$ is absolutely abelian. The genus number of $F$ is defined as $g_F = \left[ F^\ast \colon F \right]$. The genus number $g_F$ is a divisor of the narrow class number $h_F^+$, thus it is an interesting object in number theory. When $F$ is a cubic field, it is known that $g_F$ is a power of $3$ (see Theorem~\ref{genus}). 

McGown and Tucker \cite{mt} obtained the average size of the genus number for cubic fields. In their proof, they used as an ingredient a uniform estimate for the error term in the count of cubic fields given by Taniguchi and Thorne \cite{tt}. Recently, this uniform estimate was significantly improved by Bhargava, Taniguchi and Thorne \cite{btt}. The purpose of this paper is to improve the results of McGown and Tucker \cite{mt} using \cite{btt}, and in particular, to demonstrate the existence of the second order main terms. 

For simplicity, set
\begin{equation*}
    a_p \coloneqq \frac{1}{p \left( p+1 \right)}, \quad 
    b_p \coloneqq \frac{1}{p^2 + p^{4/3} + p + p^{2/3}}. 
\end{equation*}
Then we will prove:

\begin{thm}\label{mtkairyou3}
    We have
    \begin{equation}
        \sum_{\substack{ \left[ F \colon \mathbb{Q} \right] = 3 \\ 0< \pm \disc \left( F \right) < X}} g_F
        = \alpha^{\pm} X + \beta^{\pm} X^{5/6} + O_\varepsilon \left( X^{2/3 + \varepsilon} \right),
    \end{equation}
    where the sum is taken over isomorphism classes of cubic fields $F$ with $0< \pm \disc \left( F \right) < X$, and
    \begin{align*}
        \alpha^{\pm} &\coloneqq
        \frac{119 C^{\pm}}{1404 \zeta \left( 2 \right)}
        \prod_{p}
        \left( 1 + n_p a_p \right),\\
        \beta^{\pm} &\coloneqq
        \frac{2 \left( 55 \sqrt[3]{9} -19 \right) K^{\pm} \zeta \left( 1/3 \right)}{45 \left( 3\sqrt[3]{9} - 1 \right) {\Gamma \left( 2/3 \right)}^3}
        \prod_p \left( 1- \frac{p^{1/3} +1}{p \left( p+1 \right)} \right)
        \left( 1 + n_p b_p \right).
    \end{align*}
    Here, $C^{+} \coloneqq 1,C^{-} \coloneqq 3,K^{+} \coloneqq 1$, and $K^{-}  \coloneqq \sqrt{3}$, the products are taken over all primes $p$, and $n_p$ is $3$ when $p \equiv 1 \pmod{3}$ and $1$ otherwise.
\end{thm}

In \cite{mt}, they obtained the primary term $\alpha^{\pm} X$ with the error term estimate $O_\varepsilon \left( X^{16/17 + \varepsilon} \right)$. They also studied the number of cubic fields with given genus number. We improve the result also. 
For a squarefree integer $f$, we set $a_f \coloneqq \prod_{p \mid f} a_p$ and $b_f \coloneqq \prod_{p \mid f} b_p$. 
Let $\SF_k$ denote the set of all squarefree positive integers which are coprime to $3$ and have exactly $k$ prime factors $p$ satisfying $p \equiv 1 \pmod{3}$. 
For an integer $k \ge 0$, we set
\begin{equation*}
    A_k \coloneqq \sum_{f \in \SF_k} a_f, \quad
    B_k \coloneqq \sum_{f \in \SF_k} b_f,
\end{equation*}
where the bounds $a_f, b_f \le f^{-2}$ assure the convergences of these infinite sums. 
We show:

\begin{thm}\label{mtkairyou1}
    For an integer $k \ge 0$, let $N_k^{\pm} \left( X \right)$ denote the number of isomorphism classes of cubic fields $F$ satisfying $0 < \pm \disc \left( F \right) < X$ and $g_F = 3^k$. Then
    \begin{equation}
        N^{\pm}_k \left( X \right) = \alpha_k^{\pm} X 
        + \beta_k^{\pm} X^{5/6} + E_k^{\pm} \left( X \right), 
    \end{equation}
    where
    \begin{align*}
        \alpha_k^{\pm}
        &\coloneqq
        \frac{C^{\pm}}{324 \zeta \left( 2 \right)} 
        \left( 29 A_k
        + \frac{1}{4} A_{k-1} \right),\\
        \beta_k^{\pm}
        &\coloneqq
        \frac{4 \zeta \left( 1/3 \right) K^{\pm}}{45 \left( 11 \sqrt[3]{9} - 3 \right) {\Gamma \left( 2/3 \right)}^3}
        \prod_p \left( 1- \frac{p^{1/3} +1}{p \left( p+1 \right)} \right) \left( \left( 107 \sqrt[3]{9} -35 \right)
        B_k + \left( \sqrt[3]{9} -1 \right)
        B_{k-1} \right),
    \end{align*}
    with the following bound on the sum of the error terms $E^{\pm}_k \left( \Sigma \right)$:
    \begin{equation}
        \sum_{k \ge 0} \left| E_k^{\pm} \left( X \right) \right| \ll_\varepsilon X^{2/3 + \varepsilon}.
    \end{equation}
\end{thm}

In \cite{mt}, they obtained the primary term $\alpha_k^{\pm} X$ with an error term estimate $E^{\pm}_k \left( X \right) = O_\varepsilon \left( X^{16/17 + \varepsilon} \right)$. 

Although we do not show Theorem~\ref{mtkairyou3} as a consequence of Theorem~\ref{mtkairyou1}, the coefficients of the main terms satisfy the relations
\begin{equation*}
    \alpha^{\pm} = \sum_{k \ge 0} 3^k \alpha_k^{\pm}, \quad
    \beta^{\pm} = \sum_{k \ge 0} 3^k \beta_k^{\pm},
\end{equation*}
which reflect the relation
\begin{equation*}
    \sum_{\substack{\left[ F \colon \mathbb{Q} \right] = 3 \\ 0< \pm \disc \left( F \right) < X}} g_F
    = \sum_{k \ge 0} 3^k N^{\pm}_k \left( X \right) 
\end{equation*}
between the counting functions. 

Theorems~\ref{mtkairyou3} and \ref{mtkairyou1} can be viewed as special cases of our more general results with local specifications. We consider, for a place $v$ of $\mathbb{Q}$ (i.e. either a prime $p$ or the infinite place $\infty$), a set $\Sigma_v$ of isomorphism classes of \'{e}tale cubic algebras over $\mathbb{Q}_v$. We call $\Sigma_p$ a \textit{local specification} at $v$, and say that a cubic field $F$ satisfies $\Sigma_v$ if $F \otimes_{\mathbb{Q}} \mathbb{Q}_v \in \Sigma_v$. 

Note that a cubic \textit{splitting type} at a prime $p$ is also a local specification at $p$. There are five ways in which a prime can split in a cubic field; totally split, partially split, inert, partially ramified, and totally ramified cases, denoted by $( 111 )$, $( 12 )$, $( 3 )$, $( 1^2 1 )$, and $( 1^3 )$, respectively. Let $\mathscr{A}_p$ be the set of all these classes, and let $\mathscr{T}_p$ be the set consisting of the totally ramified type $(1^3)$. We write $\mathscr{A}'_p \coloneqq \mathscr{A}_p \setminus \mathscr{T}_p$. We say that $\Sigma_p$ is \textit{ordinary} if $\Sigma_p = \mathscr{A}_p$ or $\Sigma_p = \mathscr{A}'_p$. Let $\Sigma = \left( \Sigma_v \right)_v$ be a collection of local specifications where $v$ runs over all places. We say that $\Sigma = \left( \Sigma_v \right)_v$ is \textit{almost ordinary} if $\Sigma_p$ is ordinary at $p$ for all but finitely many primes $p$. By abuse of notation, we also write $\Sigma = \left( \Sigma_v \right)_v$ for the set of cubic fields $F$ satisfying $\Sigma_v$ at every place $v$. 

In this paper, we impose the following conditions on the collection of local specifications $\Sigma = \left( \Sigma_v \right)_v$: 
\begin{itemize}
    \item For all primes $p$ other than $3$, $\Sigma_p$ is a union of splitting types, 
    \item For all primes $p$ other than $3$, $\Sigma_p$ is either $\mathscr{A}_p$, $\mathscr{T}_p$, or a subset of $\mathscr{A}'_p$,
    \item For the infinite place $\infty$, $\Sigma_\infty$ is a singleton: either $\left\{ \mathbb{R}^3 \right\}$ or $\left\{ \mathbb{R} \times \mathbb{C} \right\}$. 
\end{itemize}
We impose the second condition because the genus number $g_F$ of a cubic field $F$ is determined by whether each prime $p \equiv 1 \pmod{3}$ has the totally ramified type $(1^3)$ in $F$. 

For an almost ordinary collection of local specifications $\Sigma = \left( \Sigma_v \right)_v$ satisfying the conditions above, let $c \left( \Sigma \right)$ denote the product of all primes $p \neq 3$ such that $\Sigma_p$ is not ordinary and that $\Sigma_p \neq \mathscr{T}_p$, and let $l = l \left( \Sigma \right)$ denote the number of primes $p \equiv 1 \pmod{3}$ such that $\Sigma_p = \mathscr{T}_p$. 

We now state our main theorems for almost ordinary local specifications. 
\begin{thm}\label{main}
    Let $k \ge 0$ be an integer and let $\Sigma = \left( \Sigma_v \right)_v$ be an almost ordinary collection of local specifications satisfying the conditions above. Let $N_k \left( X,\Sigma \right)$ denote the number of isomorphism classes of cubic fields $F$ satisfying $\left| \disc \left( F \right) \right| < X$, $F \in \Sigma$, and $g_F = 3^k$. 
    Then there exist real constants $\alpha^{\Sigma}_k$ and $\beta^{\Sigma}_k$ described explicitly such that
    \begin{equation}\label{mainsiki}
        N_k \left( X , \Sigma \right)
        = \alpha^{\Sigma}_k X 
        + \beta^{\Sigma}_k X^{5/6}
        + E_k \left( X , \Sigma \right),
    \end{equation}
    with the following bound on the sum of the error terms $E_k \left( X, \Sigma \right)$:
    \begin{equation}
        \sum_{k \ge l} \left| E_k \left( X , \Sigma \right) \right| \ll_{\varepsilon} {c \left( \Sigma \right)}^{2/3} X^{2/3 + \varepsilon}.
    \end{equation}
\end{thm}

For the explicit descriptions of $\alpha^{\Sigma}_k$ and $\beta^{\Sigma}_k$, see \eqref{mainkeisuu11} and \eqref{mainkeisuu12} in Section~\ref{pmr}. 
Since $N_k \left( X, \Sigma \right) =0$ for $k< l \left( \Sigma \right)$, the estimate is only valuable for $k \ge l \left( \Sigma \right)$. 

We also give uniform estimates for the moments of the genus number. 

\begin{thm}\label{main2}
    Let $\Sigma = \left( \Sigma_v \right)_v$ be an almost ordinary collection of local specifications satisfying the conditions above. Then for any $\varepsilon >0$ and any positive function $G \left( X \right) = o \left( \log {\log {X}} \right)$, there is a constant $C_{G,\varepsilon}$ that depends on $G$ and $\varepsilon$, such that for all $X > C_{G,\varepsilon}$, we have the following uniform estimate in the region $\Re \left( z \right) \le G \left( X \right)$:
    \begin{equation}\label{moment}
        \sum_{\substack{\left| \disc \left( F \right) \right| < X \\ F \in \Sigma}} {g_F}^z
        = \alpha^{\Sigma} \left( z \right) X + \beta^{\Sigma} \left( z \right) X^{5/6} + O_{\varepsilon} \left( 3^{l G \left( X \right)} {c \left( \Sigma \right)}^{2/3} X^{2/3 + \varepsilon} \right),
    \end{equation}
    where 
    \begin{equation*}
        \alpha^{\Sigma} \left( z \right) = \sum_{k \ge l} 3^{kz} \alpha^{\Sigma}_k , \quad
        \beta^{\Sigma} \left( z \right) = \sum_{k \ge l} 3^{kz} \beta^{\Sigma}_k,
    \end{equation*}
    and the implied constant depends only on $\varepsilon$, not on $z$ and $G$. 
\end{thm}

Alternative descriptions of $\alpha^{\Sigma} \left( z \right)$ and $\beta^{\Sigma} \left( z \right)$ are given in \eqref{mainkeisuu21} and \eqref{mainkeisuu22} in Section~\ref{pmr}. 
We note that the case $z=0$ recovers Theorem~1.3 of \cite{btt} and the case $z=1$ implies our Theorem~\ref{mtkairyou3}. 

In these theorems, it is possible to remove the second bullet assumption 
"For all primes $p$ other than $3$, $\Sigma_p$ is either $\mathscr{A}_p$, $\mathscr{T}_p$, or a subset of $\mathscr{A}'_p$". 
For that, we instead consider the partition $\Sigma_p = \mathscr{T}_p \cup \left( \Sigma_p \cap \mathscr{A}'_p \right)$ for each $p \mid c \left( \Sigma \right)$. We then apply Theorems~\ref{main} and \ref{main2} to each the $\le 2^{\omega \left( c \left( \Sigma \right) \right)}$ resulting local specifications and sum their asymptotic formulas. This yields asymptotic formulas with error terms of $\ll_{\varepsilon} {c \left( \Sigma \right)}^{2/3 + \varepsilon} X^{2/3 + \varepsilon}$ and $O_{\varepsilon} \left( 3^{l' G \left( X \right)} {c \left( \Sigma \right)}^{2/3 + \varepsilon} X^{2/3 + \varepsilon} \right)$, respectively, where $l' = l' \left( \Sigma \right) $ is the number of primes $p \equiv 1 \pmod{3}$ such that $\mathscr{T}_p \subseteq \Sigma_p \subsetneq \mathscr{A}_p$. We can of course remove the third assumption that $\Sigma_\infty$ is a singleton. 

Our proof largely follows the strategy of McGown and Tucker \cite{mt}. The improved uniform estimates of \cite{btt} allow us to obtain the improved error estimates in our main theorems. 

We mention related works on genus numbers of number fields. Similar questions have been studied for other families of number fields. The distribution of genus numbers for abelian fields is analyzed in \cite{fln}, while the counting of quintic fields with a given genus number is treated in \cite{mtt}. 
More recently, Bose, McGown, Panpaliya, Welling, and Williams \cite{bmpww} studied the average size of the genus number for pure fields of prime degree. To our best knowledge, our result is the first example in which a secondary term with a lower exponent of $X$ than the primary term was obtained in the kind of statistics of genus numbers. 

The organization of this paper is as follows:
In Section~\ref{sec2}, we summarize the necessary results on the density of cubic fields. In particular, we state the uniform error term estimates for counting cubic fields with prescribed local specifications, which are the results in \cite{btt}. 
In Section~\ref{pmr}, we provide the proofs of our main results, Theorems~1.3 and 1.4. 
We first recall Fr\"{o}hlich's formula, which shows that the genus number is determined by local properties. Then, 
using the discriminant decomposition we analyze the counting functions for cubic fields with fixed squarefree parts of the discriminant. By summing these contributions, we derive the asymptotic formulas for the distribution of genus numbers.

\section{Preliminaries on Cubic Field Densities}\label{sec2}
In studying the distribution of cubic fields satisfying given local conditions, the following two theorems (Theorem~\ref{btt+a1} and Theorem~\ref{btt+a2}) serve as our main tools. Theorems~\ref{btt+a1} and \ref{btt+a2} are Theorems~1.3 and 1.4 in \cite{btt}, respectively. 

\begin{thm}[Bhargava-Taniguchi-Thorne]\label{btt+a1}
    For an almost ordinary collection of local specifications $\Sigma = \left( \Sigma_v \right)_v$, let $N \left( X , \Sigma \right)$ denote the number of $F \in \Sigma$ satisfying $\left| \disc \left( F \right) \right| \le X$
    Then, we have  
    \begin{equation}
        N \left( X , \Sigma \right) =
        C \left( \Sigma \right) \frac{1}{12 \zeta \left( 3 \right)} X
        + K \left( \Sigma \right) \frac{4 \zeta \left( 1/3 \right)}{5 {\Gamma \left( 2/3 \right)}^3 \zeta \left( 5/3 \right)} X^{5/6}
        + O_{\varepsilon} \left( E \left( \Sigma \right) X^{2/3 + \varepsilon} \right)
    \end{equation}
    where the constants $C \left( \Sigma \right)$, $K \left( \Sigma \right)$, and $E \left( \Sigma \right)$ are defined below.
\end{thm}

The constants $C \left( \Sigma \right)$, $K \left( \Sigma \right)$, and $E \left( \Sigma \right)$ are defined by 
\begin{equation}
    C \left( \Sigma \right) \coloneqq \prod_{v} C_v \left( \Sigma_v \right), \quad
    K \left( \Sigma \right) \coloneqq \prod_{v} K_v \left( \Sigma_v \right), \quad
    E \left( \Sigma \right) \coloneqq \prod_v E_v \left( \Sigma_v \right),
\end{equation}
where $v$ runs over all places of $\mathbb{Q}$. 
The local factors $C_{\infty} \left( \Sigma_{\infty} \right)$, $K_{\infty} \left( \Sigma_{\infty} \right)$, and $E_\infty \left( \Sigma_\infty \right)$ are defined by $C_{\infty} \left( \mathbb{R}^3 \right) \coloneqq C^{+}$, $C_{\infty} \left( \mathbb{R} \times \mathbb{C} \right) \coloneqq C^{-}$, $K_{\infty} \left( \mathbb{R}^3 \right) \coloneqq K^{+}$, $K_{\infty} \left( \mathbb{R} \times \mathbb{C} \right) \coloneqq K^{-}$, and $E_\infty \left( \Sigma_\infty \right) \coloneqq 1$. 
The local factors $C_p \left( \Sigma_p \right)$ and $K_p \left( \Sigma_p \right)$ for a prime $p$ are defined as in Theorem~1.3 of \cite{btt}. 

Let $U$ be any squarefree positive integer, and for each $p \mid U$, let $\Sigma_p$ be an arbitrary local specification such that every algebra in $\Sigma_p$ has the same ramification type at $p$. For each pair $r,s$ of positive squarefree integers such that $\left( rs , U \right) =1$, we complete this to a collection $\Sigma^{r,s} = \left( \Sigma_v \right)_v$ of local specifications over all $v$ satisfying the following condition:
\begin{itemize}
    \item If $v \mid r$, then $\Sigma_v$ consists of all partially ramified cubic extensions of $\mathbb{Q}_v$.
    \item If $v \mid s$, then $\Sigma_v$ consists of all totally ramified cubic extensions of $\mathbb{Q}_v$.
    \item If $v \nmid Urs$, then $\Sigma_v$ is ordinary, i.e., $\Sigma_v = \mathscr{A}_v$ or $\mathscr{A}'_v$.
    \item If $v = \infty$, then $\Sigma_v = \left\{ \mathbb{R}^3 \right\}$ or $\left\{ \mathbb{R} \times \mathbb{C} \right\}$.
\end{itemize}

\begin{thm}[Bhargava-Taniguchi-Thorne]\label{btt+a2}
    For each collection of cubic local specifications $\Sigma = \left( \Sigma_v \right)_v$, let $E \left( X, \Sigma \right)$ denote the error term in estimating $N \left( X, \Sigma \right)$ in Theorem~\ref{btt+a1}. Then for any $\varepsilon >0$, we have
    \begin{equation}
        \sum_{r \le R} \sum_{s} \left| E \left( X_{r,s} , \Sigma^{r,s} \right) \right|
        \ll_{\varepsilon}
        X^{2/3 + \varepsilon} R^{2/3} \prod_{p \mid U} E_p \left( \Sigma_p \right)
    \end{equation}
    for any $R,X > 0$ and any $X_{r,s}$ with $X_{r,s} \le X$; here the sum is over all squarefree integers $r$ and $s$ with $\left( U , rs \right) =1$ and $r \le R$. 
\end{thm}
This completes the preliminary setup for the proofs in the subsequent sections. 

\section{Proofs of Main Results}\label{pmr}
In this section, we prove Theorems~\ref{main} and \ref{main2}. For the rest of this paper, we fix an almost ordinary collection of local specifications $\Sigma = \left( \Sigma_v \right)_v$ satisfying the three conditions in Section~\ref{sec1}. Recall that we defined $c \left( \Sigma \right)$ and $l = l \left( \Sigma \right)$. Let $t= t \left( \Sigma \right)$ denote the product of primes $p \neq 3$ such that $\Sigma_p = \mathscr{T}_p$. For a positive integer $n$, we define $\psi \left( n \right)$ to be the number of prime divisors $p$ of $n$ such that $p \equiv 1 \pmod{3}$. Then by definition, $l = \psi \left( t \right)$. Let $\mathcal{P} = \mathcal{P} \left( \Sigma \right)$ denote the set of primes $p \neq 3$ such that $\Sigma_p = \mathscr{A}_p$. 

Let $F$ be a cubic field. Then $\disc \left( F \right)$ takes one of the forms $d_F {f_F}^2$, $9 d_F {f_F}^2$, or $81d_F {f_F}^2$, where $d_F$ is a fundamental discriminant; and $f_F$ is the product of all primes $p \neq 3$ that are totally ramified in $F$. This discriminant decomposition leads to an explicit description of the genus number of a cubic field, as established by Fr\"{o}hlich (see \cite{fro}; see also \cite{ishida} for related discussion). 
\begin{thm}[Fr\"{o}hlich]\label{genus}
    Let $e_F$ denote the number of odd primes $p$ such that $p$ is totally ramified in $F$ and $\left( d_F /p \right) = 1$, where $\left( d_F /p \right)$ is the usual Legendre symbol. Then we have
    \begin{equation}
        g_F =
        \begin{dcases}
            3^{e_F -1} &\text{if } F \text{ is cyclic,}\\
            3^{e_F} &\text{if } F \text{ is not cyclic.}
        \end{dcases}
    \end{equation}
\end{thm}
The contribution of cyclic fields is negligible for the main results of this paper. 
First, regarding Theorem~\ref{main}, the number of cyclic cubic fields with discriminant up to $X$ is known to be $O \left( X^{1/2} \right)$ (see \cite[Section~2]{mt}). 
Second, regarding Theorem~\ref{main2}, we verify that the contribution of cyclic fields to the moments is negligible. Recall that for a squarefree positive integer $f$ coprime to $3$, the number of cyclic cubic fields $F$ satisfying $f_F = f$ is at most $2^{\omega \left( f \right)}$ (see \cite[Section~4]{mt}), and their genus numbers satisfy $g_F \le 3^{\omega \left( f \right)}$. Under the assumption $\Re \left( z \right) \le G \left( X \right)$ with $G \left( X \right) = o \left( \log \log X \right)$, Th\'{e}or\`{e}me~11 of \cite{Rob} implies that for any $\varepsilon > 0$, there exists a constant $C'_{G, \varepsilon}$ depending on $G$ and $\varepsilon$ such that for all $X > C'_{G, \varepsilon}$, we have
\begin{equation*}
    \sum_{\substack{\left| \disc \left( F \right) \right| <X \\ F \colon \text{cyclic}}} {g_F}^z
    \ll \sum_{f < \sqrt{X}} 2^{\omega \left( f \right)} \cdot 3^{\omega \left( f \right) G \left( X \right)}
    \ll \sum_{f < \sqrt{X}} 6^{\omega \left( f \right) G \left( X \right)}
    \ll_\varepsilon \sum_{f < \sqrt{X}} X^{\varepsilon}
    \ll X^{1/2 + \varepsilon}. 
\end{equation*}
Consequently, we restrict our attention to non-cyclic cubic fields for the remainder of this paper. 

If $F$ is a cubic field and $p \neq 2,3$ is a prime that is totally ramified in $F$, then it is known (see \cite[Example~10]{ishida}) that
\begin{equation}
    \left( \frac{d_F}{p} \right) = 1 \Longleftrightarrow p \equiv 1 \!\!\!\pmod{3},
\end{equation}
and, of course, $\left( d_F /3 \right) = 1$ equivalent to $d_F \equiv 1 \pmod{3}$. 
We proceed by classifying the fields $F \in \Sigma$ according to the values of $f_F$. For any $F \in \Sigma$, the associated integer $f_F$ is necessarily a squarefree positive integer that is divisible by $t$, and whose remaining prime factors lie entirely in $\mathcal{P}$. We refer to a squarefree positive integer, whose prime factors lie entirely in $\mathcal{P}$ as \textit{$\Sigma$-compatible}. Recall that $\SF_k$ denotes the set of all squarefree positive integers which are coprime to $3$ and have exactly $k$ prime factors $p$ satisfying $p \equiv 1 \pmod{3}$. Let $\SF \left( \Sigma \right)$ denote the set of all $\Sigma$-compatible integers $f$. For $k \ge 0$, we then define $\SF_k \left( \Sigma \right) \coloneqq \SF_k \cap \SF \left( \Sigma \right)$. For a $\Sigma$-compatible integer $f$, we define
\begin{align}
    \mathcal{F} \left( f \right)
    &\coloneqq \left\{ F \in \Sigma \mid f_F = tf \right\},\\
    \mathcal{F}' \left( f \right)
    &\coloneqq \left\{ F \in \mathcal{F} \left( f \right) \mid 3 \text{ is totally ramified in } F \text{ and } d_F \equiv 1 \!\!\!\!\pmod{3} \right\} \label{F'f}, 
\end{align}
and let $M \left( X , f \right)$ and $M' \left( X , f \right)$ denote the number of cubic fields $F$ with $\left| \disc \left( F \right) \right| < X$ satisfying $F \in \mathcal{F} \left( f \right)$ and $F \in \mathcal{F}' \left( f \right)$, respectively. 
By Theorem~\ref{genus}, we have
\begin{equation}\label{nmmm}
    N_k \left( X , \Sigma \right)
    = \sum_{f \in \SF_{k-l} \left( \Sigma \right)} M \left( X , f \right)
    - \sum_{f \in \SF_{k-l} \left( \Sigma \right)} M' \left( X , f \right)
    + \sum_{f \in \SF_{k-l-1} \left( \Sigma \right)} M' \left( X , f \right). 
\end{equation}

First, we establish the following result for $M \left( X,f \right)$. 

\begin{prop}\label{prop1}
    For $\Sigma$-compatible integer $f$, We have
    \begin{equation}\label{pro32}
        M \left( X , f \right) = \alpha \left( f \right) X + \beta \left( f \right) X^{5/6} + E \left( X, f \right),
    \end{equation}
    where
    \begin{align*}
        \alpha \left( f \right)
        &\coloneqq \frac{C_{\infty} \left( \Sigma_{\infty} \right)}{12 \zeta \left( 2 \right)} a_{tf}
        \prod_{p \mid 3c \left( \Sigma \right)} C_p \left( \Sigma_p \right) \left( 1+ a_p \right),\\
        \beta \left( f \right)
        &\coloneqq
        \frac{4 \zeta \left( 1/3 \right) K_{\infty} \left( \Sigma_{\infty} \right)}{5 {\Gamma \left( 2/3 \right)}^3} b_{tf}
        \prod_p \left( 1- \frac{p^{1/3} +1}{p \left( p+1 \right)} \right)
        \prod_{p \mid 3c \left( \Sigma \right)} K_p \left( \Sigma_p \right) \left( 1+ b_p \right),
    \end{align*}
    and for the estimate of the error term $E \left( X,f \right)$ we have the following bound: 
    \begin{equation*}
        \sum_{f \in \SF \left( \Sigma \right)} \left| E \left( X , f \right) \right| \ll_{\varepsilon} {c \left( \Sigma \right)}^{2/3} X^{2/3 + \varepsilon}.
    \end{equation*}
\end{prop}

\begin{proof}
    We define $\Sigma^f = \left( \Sigma^f_v \right)_v$ by setting
    \begin{equation}
        \Sigma^f_v \coloneqq
        \begin{dcases}
            \Sigma_v \cap \mathscr{T}_v &\text{if } v \mid tf;\\
            \Sigma_v \cap \mathscr{A}'_v &\text{if } v \nmid 3tf;\\
            \Sigma_v &\text{if } v = 3 , \infty.
        \end{dcases}
    \end{equation}
    Then we have $M \left( X , f \right) = N \left( X , \Sigma^f \right)$. 
    For convenience, we define
    \begin{equation}
        \mathcal{A}_p \coloneqq \frac{1}{p^2 + p + 1}, \quad \mathcal{B}_p \coloneqq \frac{1+p^{-1/3}}{p^2} \cdot \frac{1-p^{-1/3}}{\left( 1-p^{-5/3} \right) \left(1+p^{-1} \right)}.
    \end{equation}
    For each prime $p$, the set $\Sigma^f_p$ and the values $C_p \left( \Sigma^f_p \right)$, $K_p \left( \Sigma^f_p \right)$, and $E_p \left( \Sigma^f_p \right)$ are given as follows:
    \begin{center}
        \begin{tabular}{l|llll}
             \hline
             $p$ & $\Sigma^f_p$ & $C_p \left( \Sigma^f_p \right)$ & $K_p \left( \Sigma^f_p \right)$ & $E_p \left( \Sigma^f_p \right)$ \\
             \hline
             $p \mid 3c \left( \Sigma \right)$ & $\Sigma_p$ & $C_p \left( \Sigma_p \right)$ & $K_p \left( \Sigma_p \right)$ & $p^{2/3}$ \\
             $p \mid tf$ & $\mathscr{T}_p$ & $\mathcal{A}_p$ &
             $\mathcal{B}_p$ & 1\\
             $p \nmid 3 c \left( \Sigma \right) tf$ & $\mathscr{A}'_p$ & $1- \mathcal{A}_p$ & $1- \mathcal{B}_p$ & 1\\
             \hline
        \end{tabular}
    \end{center}
    Hence, we have
    \begin{align*}
        \prod_p C_p \left( \Sigma^f_p \right)
        &= \prod_{p \mid 3 c \left( \Sigma \right)} C_p \left( \Sigma_p \right) \prod_{p \mid tf} \mathcal{A}_p \prod_{p \nmid 3 c \left( \Sigma \right) tf} \left( 1- \mathcal{A}_p \right)\\
        &= \prod_{p \mid  3 c \left( \Sigma \right)} C_p \left( \Sigma_p \right) \left( 1+ \frac{\mathcal{A}_p}{1- \mathcal{A}_p} \right) \prod_{p \mid tf} \frac{\mathcal{A}_p}{1- \mathcal{A}_p} \prod_p \left( 1- \mathcal{A}_p \right).
    \end{align*}
    Since
    \begin{equation}
        1- \mathcal{A}_p = \frac{1- p^{-2}}{1- p^{-3}} , ~ 
        \frac{\mathcal{A}_p}{1- \mathcal{A}_p} = a_p,
    \end{equation}
    we obtain
    \begin{equation}
        C \left( \Sigma^f \right) =
        C_{\infty} \left( \Sigma_{\infty} \right) \frac{\zeta \left( 3 \right)}{\zeta \left( 2 \right)} a_{tf}
        \prod_{p \mid 3 c \left( \Sigma \right)} C_p \left( \Sigma_p \right) \left( 1+ a_p \right) = \alpha \left( f \right).
    \end{equation}
    Similarly, from
    \begin{equation}
        1- \mathcal{B}_p = \frac{1}{1- p^{-5/3}} \left( 1- \frac{p^{1/3}+1}{p \left( p+1 \right)} \right) , ~ 
        \frac{\mathcal{B}_p}{1- \mathcal{B}_p} = b_p,
    \end{equation}
    we obtain
    \begin{equation*}
        \prod_p K_p \left( \Sigma^f_p \right)
        = \zeta \left( 5/3 \right) b_{tf}
        \prod_p \left( 1- \frac{p^{1/3} +1}{p \left( p+1 \right)} \right)
        \prod_{p \mid 3 c \left( \Sigma \right)} K_p \left( \Sigma_p \right) \left( 1+ b_p \right),
    \end{equation*}
    and thus $K \left( \Sigma^f \right) = \beta \left( f \right)$. Therefore we obtain the main terms of \eqref{pro32}. 

    For primes $p \mid 3 c \left( \Sigma \right) t$, the local specifications $\Sigma^f_p$ do not depend on $f$. For primes $p \nmid 3 c \left( \Sigma \right) t$, the local specification $\Sigma^f_p$ is taking the form $\mathscr{T}_p$ if $p \mid f$, and $\mathscr{A}'_p$ otherwise. Therefore, by Theorem~\ref{btt+a2}, we obtain
    \begin{equation}
        \sum_{f \in \SF \left( \Sigma \right)} \left| E \left( X , f \right) \right|
        = \sum_{f \in \SF \left( \Sigma \right)} \left| E \left( X , \Sigma^f \right) \right|
        \ll_{\varepsilon} X^{2/3 + \varepsilon} \prod_{p \mid 3 c \left( \Sigma \right) t} E \left( \Sigma_p \right)
        \ll_{\varepsilon} {c \left( \Sigma \right)}^{2/3} X^{2/3 + \varepsilon}.
    \end{equation}
\end{proof}

Next, we consider $M' \left( X,f \right)$. 
According to \cite[Lemma~2.1]{tt3}, for a cubic field $F$, 
\begin{equation*}
    \begin{dcases}
        3 \text{ is totally ramified in } F,\\
        d_F \equiv 1 \pmod{3}
    \end{dcases}
    \Longleftrightarrow
    \begin{dcases}
        \disc_3 \left( F_3 \right) = 3^4,\\
        \frac{\disc \left( F_3 \right)}{\disc_3 \left( F_3 \right)} \equiv 1 \pmod{3},
    \end{dcases}
\end{equation*}
where we note that the quotient $\disc \left( F_3 \right)/\disc_3 \left( F_3 \right)$ is well-defined modulo $3$ because $u^2 \equiv 1 \pmod{3}$ for any $u \in \mathbb{Z}_3^{\times}$. 
Let $\mathscr{T}_3^{\left( 1 \right)}$ denote the subset of $\mathscr{A}_3$ consisting of cubic \'{e}tale $\mathbb{Q}_3$-algebras $F_3$ such that $\disc_3 \left( F_3 \right) = 3^4$ and $\disc \left( F_3 \right)/\disc_3 \left( F_3 \right) \equiv 1 \pmod{3}$. According to \cite[Section~6.2]{tt}, $\mathscr{T}_3^{\left( 1 \right)}$ is the set of three totally ramified cubic \'{e}tale $\mathbb{Q}_3$-algebras with the generating polynomials $x^3 - 3 x^2 + 3u$ for $u \in \left\{ 1,4,7 \right\}$. 

\begin{prop}\label{prop2}
    For $\Sigma$-compatible integer $f$, We have
    \begin{equation}
        M' \left( X , f \right) = \alpha' \left( f \right) X + \beta' \left( f \right) X^{5/6} + E' \left( X,f \right),
    \end{equation}
    where
    \begin{align*}
        \alpha' \left( f \right)
        &\coloneqq \frac{\left( 1+a_3 \right) C_{\infty} \left( \Sigma_{\infty} \right) C_3 \left( \Sigma_3 \cap \mathscr{T}_3^{\left( 1 \right)} \right) }{12 \zeta \left( 2 \right)} a_{tf}
        \prod_{p \mid c \left( \Sigma \right)} C_p \left( \Sigma_p \right) \left( 1+ a_p \right),\\
        \beta' \left( f \right)
        &\coloneqq
        \frac{4 \left( 1+ b_3 \right) \zeta \left( 1/3 \right) K_{\infty} \left( \Sigma_{\infty} \right) K_3 \left( \Sigma_3 \cap \mathscr{T}_3^{\left( 1 \right)} \right)}{5 {\Gamma \left( 2/3 \right)}^3} b_{tf}
        \prod_p \left( 1- \frac{p^{1/3} +1}{p \left( p+1 \right)} \right)
        \prod_{p \mid c \left( \Sigma \right)} K_p \left( \Sigma_p \right) \left( 1+ b_p \right),
    \end{align*}
    and for the estimate of the error term $E' \left( X,f \right)$ we have the following bound:
    \begin{equation*}
        \sum_{f \in \SF \left( \Sigma \right)} \left| E' \left( X , f \right) \right| \ll_{\varepsilon} {c \left( \Sigma \right)}^{2/3} X^{2/3 + \varepsilon}.
    \end{equation*}
\end{prop}

\begin{proof}
    We define $\Sigma^{3,f} = \left( \Sigma^{3,f}_v \right)_v$ by setting
    \begin{equation}
        \Sigma^{3,f}_v \coloneqq
        \begin{dcases}
            \Sigma_v \cap \mathscr{T}_v &\text{if } v \mid tf;\\
            \Sigma_v \cap \mathscr{A}'_v &\text{if } v \nmid 3tf;\\
            \Sigma_v &\text{if } v = \infty;\\
            \Sigma_3 \cap \mathscr{T}_3^{\left( 1 \right)} &\text{if } v = 3.
        \end{dcases}
    \end{equation}
    Then we have $M' \left( X , f \right) = N \left( X , \Sigma^{3,f} \right)$.
    For all $v \neq 3$, we have $\Sigma^{3,f}_v = \Sigma^f_v$. Therefore, as in the previous proposition, $C \left( \Sigma^{3,f} \right)$, $K \left( \Sigma^{3,f} \right)$, and $E \left( \Sigma^{3,f} \right)$ are computed accordingly, and the proof of the first part is complete. 
    
    For primes $p \mid 3 c \left( \Sigma \right) t$, the local specifications $\Sigma^{3,f}_p$ do not depend on $f$. For primes $p \nmid 3 c \left( \Sigma \right) t$, the local specification $\Sigma^{3,f}_p$ is taking the form $\mathscr{T}_p$ if $p \mid tf$, and $\mathscr{A}'_p$ otherwise. Therefore, by Theorem~\ref{btt+a2}, we obtain
    \begin{equation}
        \sum_{f \in \SF \left( \Sigma \right)} \left| E' \left( X , f \right) \right|
        = \sum_{f \in \SF \left( \Sigma \right)} \left| E \left( X , \Sigma^{3,f} \right) \right|
        \ll_{\varepsilon} c \left( \Sigma \right)^{2/3} X^{2/3 + \varepsilon}.
    \end{equation}
\end{proof}

We are now ready to prove our main theorem. We begin with the proof of Theorem~\ref{main}. 
\begin{proof}[Proof of Theorem~\textnormal{\ref{main}}]
    We define $\alpha^{\Sigma}_k$ and $\beta^{\Sigma}_k$ as
    \begin{align*}
        \alpha^{\Sigma}_k &\coloneqq
        \sum_{f \in \SF_{k-l} \left( \Sigma \right)} \left( \alpha \left( f \right)
        - \alpha' \left( f \right) \right)
        + \sum_{f \in \SF_{k-l-1} \left( \Sigma \right) } \alpha' \left( f \right),\\
        \beta^{\Sigma}_k &\coloneqq
        \sum_{f \in \SF_{k-l} \left( \Sigma \right) } \left( \beta \left( f \right)
        - \beta' \left( f \right) \right)
        + \sum_{f \in \SF_{k-l-1} \left( \Sigma \right) } \beta' \left( f \right).
    \end{align*}
    Then by applying Propositions~\ref{prop1} and \ref{prop2} to \eqref{nmmm}, we have
    \begin{equation*}
        N_k \left( X, \Sigma \right) = \alpha^{\Sigma}_k X + \beta^{\Sigma}_k X^{5/6}
        + E_k \left( X , \Sigma \right),
    \end{equation*}
    where
    \begin{equation*}
        E_k \left( X , \Sigma \right)
        \coloneqq
        \sum_{f \in \SF_{k-l} \left( \Sigma \right) } E \left( X , f \right)
        - \sum_{f \in \SF_{k-l} \left( \Sigma \right)} E' \left( X , f \right)
        + \sum_{f \in \SF_{k-l-1} \left( \Sigma \right) } E' \left( X , f \right).
    \end{equation*}
    By Propositions~\ref{prop1} and \ref{prop2}, we have
    \begin{equation*}
        \sum_{k \ge l} \left| E_k \left( X, \Sigma \right) \right|
        \le \sum_{f \in \SF \left( \Sigma \right)} \left| E \left( X , f \right) \right|
        + \sum_{f \in \SF \left( \Sigma \right)} \left| E' \left( X , f \right) \right|
        + \sum_{f \in \SF \left( \Sigma \right)} \left| E' \left( X , f \right) \right|
        \ll {c \left( \Sigma \right)}^{2/3} X^{2/3 + \varepsilon}.
    \end{equation*}
    This completes the proof of Theorem~\ref{main}. 
\end{proof}

The constants $\alpha^{\Sigma}_k$ and $\beta^{\Sigma}_k$ are described as
\begin{align}
    \alpha^{\Sigma}_k
    &= \frac{\left( 1+a_3 \right) C_{\infty} \left( \Sigma_{\infty} \right) }{12 \zeta \left( 2 \right)} a_{t} I_3 \left( k,\Sigma \right)
    \prod_{p \mid c \left( \Sigma \right)} C_p \left( \Sigma_p \right) \left( 1+ a_p \right), \label{mainkeisuu11}\\
    \beta^{\Sigma}_k
    &= \frac{4 \left( 1+ b_3 \right) \zeta \left( 1/3 \right) K_{\infty} \left( \Sigma_{\infty} \right) }{5 {\Gamma \left( 2/3 \right)}^3} b_{t} J_3 \left( k, \Sigma \right)
    \prod_p \left( 1- \frac{p^{1/3} +1}{p \left( p+1 \right)} \right)
    \prod_{p \mid c \left( \Sigma \right)} K_p \left( \Sigma_p \right) \left( 1+ b_p \right), \label{mainkeisuu12}
\end{align}
where
\begin{align*}
    I_3 \left( k,\Sigma \right)
    &= C_3 \left( \Sigma_3 \setminus \mathscr{T}_3^{\left( 1 \right)} \right) A_{k-l} \left( \Sigma \right)
    + C_3 \left( \Sigma_3 \cap \mathscr{T}_3^{\left( 1 \right)} \right) A_{k-l-1} \left( \Sigma \right),\\
    J_3 \left( k,\Sigma \right)
    &= K_3 \left( \Sigma_3 \setminus \mathscr{T}_3^{\left( 1 \right)} \right) B_{k-l} \left( \Sigma \right)
    + K_3 \left( \Sigma_3 \cap \mathscr{T}_3^{\left( 1 \right)} \right) B_{k-l-1} \left( \Sigma \right). 
\end{align*}
Here, for any integer $k \ge 0$, 
\begin{equation*}
    A_k \left( \Sigma \right) \coloneqq \sum_{f \in \SF_k \left( \Sigma \right)} a_f , \quad
    B_k \left( \Sigma \right) \coloneqq \sum_{f \in \SF_k \left( \Sigma \right)} b_f. 
\end{equation*}

If we set $\Sigma_p = \mathscr{A}_p$ for all primes $p$, we obtain Theorem~\ref{mtkairyou1}. 

Finally, we turn to the proof of Theorem~\ref{main2}. 
\begin{proof}[Proof of Theorem~\textnormal{\ref{main2}}]
    Fix $\varepsilon >0$ and a positive function $G \left( X \right) = o \left( \log \log X \right)$, and assume that $z$ satisfies $\Re \left( z \right) \le G \left( X \right)$. 
    We have 
    \begin{align*}
        &\quad \sum_{\substack{\left| \disc \left( F \right) \right|< X \\ F \in \Sigma}} {g_F}^z
        =\sum_{k \ge l} 3^{kz} N_k \left( X,\Sigma \right)\\
        &= \sum_{k \ge l} 3^{kz} \sum_{f \in \SF_{k-l} \left( \Sigma \right) } M \left( X , f \right)
        - \sum_{k \ge l} 3^{kz} \sum_{f \in \SF_{k-l} \left( \Sigma \right) } M' \left( X , f \right)
        + \sum_{k \ge l} 3^{kz} \sum_{f \in \SF_{k-l-1} \left( \Sigma \right) } M' \left( X , f \right)\\
        &= \sum_{f \in \SF \left( \Sigma \right) } 3^{\left( \psi \left( f \right) + l \right)z}
        M \left( X , f \right)
        - \sum_{f \in \SF \left( \Sigma \right) } 3^{\left( \psi \left( f \right) + l \right)z} M' \left( X , f \right)
        + \sum_{f \in \SF \left( \Sigma \right) } 3^{\left( \psi \left( f \right) + l +1 \right)z} M' \left( X , f \right)\\
        &= \sum_{f \in \SF \left( \Sigma \right) } 3^{\left( \psi \left( f \right) + l \right)z} M \left( X , f \right)
        + \left( 3^z -1 \right) \sum_{f \in \SF \left( \Sigma \right) } 3^{\left( \psi \left( f \right) + l \right)z} M' \left( X , f \right)\\
        &= \sum_{\substack{f < \sqrt{X} \\ f \in \SF \left( \Sigma \right) }} 3^{\left( \psi \left( f \right) + l \right)z} M \left( X , f \right)
        + \left( 3^z -1 \right) \sum_{\substack{f < \sqrt{X} \\ f \in \SF \left( \Sigma \right) }} 3^{\left( \psi \left( f \right) + l \right)z} M' \left( X , f \right).
    \end{align*}
    The last equality is justified because $M \left( X , f \right) = M' \left( X , f \right) = 0$ for $f \ge \sqrt{X}$. We introduce this truncation to bound the contribution of the weights $3^{\psi \left( f \right) z}$ in the tail of the series in the above derivation. 
    Defining $\alpha^\Sigma \left( z \right)$ and $\beta^\Sigma \left( z \right)$ as
    \begin{equation*}
        \alpha^\Sigma \left( z \right) \coloneqq
        \sum_{f \in \SF \left( \Sigma \right)} \left( 3^{\left( \psi \left( f \right) + l \right)z} \alpha \left( f \right) +\left( 3^z -1 \right)3^{\left( \psi \left( f \right) + l \right)z} \alpha' \left( f \right)   \right),
    \end{equation*}
    and
    \begin{equation*}
        \beta^{\Sigma} \left( z \right) \coloneqq
        \sum_{f \in \SF \left( \Sigma \right)} \left( 3^{\left( \psi \left( f \right) + l \right)z} \beta \left( f \right) +\left( 3^z -1 \right)3^{\left( \psi \left( f \right) + l \right)z} \beta' \left( f \right) \right),
    \end{equation*}
    respectively, we obtain
    \begin{equation*}
        \sum_{\substack{\left| \disc \left( F \right) \right|< X \\ F \in \Sigma}} {g_F}^z
        = \alpha^\Sigma \left( z \right) X + \beta^\Sigma \left( z \right) X^{5/6}
        - E_1 + E_2,
    \end{equation*}
    where
    \begin{align*}
        E_1
        &\coloneqq
        \begin{aligned}[t]
            \sum_{\substack{f \ge \sqrt{X} \\ f \in \SF \left( \Sigma \right) }} &\left( 3^{\left( \psi \left( f \right) + l \right)z} \alpha \left( f \right) +\left( 3^z -1 \right)3^{\left( \psi \left( f \right) + l \right)z} \alpha' \left( f \right)   \right)X\\
            &+ \sum_{\substack{f \ge \sqrt{X} \\ f \in \SF \left( \Sigma \right) }} \left( 3^{\left( \psi \left( f \right) + l \right)z} \beta \left( f \right) +\left( 3^z -1 \right)3^{\left( \psi \left( f \right) + l \right)z} \beta' \left( f \right) \right)X^{5/6},
        \end{aligned}
        \\
        E_2
        &\coloneqq
        \sum_{\substack{f < \sqrt{X} \\ f \in \SF \left( \Sigma \right) }} 3^{\left( \psi \left( f \right) + l \right)z} E \left( X , f \right) + \sum_{\substack{f < \sqrt{X} \\ f \in \SF \left( \Sigma \right) }} \left( 3^z -1 \right)3^{\left( \psi \left( f \right) + l \right)z} E' \left( X, f \right).
    \end{align*}
    According to \cite[Th\'{e}or\`{e}me~11]{Rob}, for $f \ge 3$, 
    \begin{equation}
        \psi \left( f \right) \le \omega \left( f \right) \le 1.38402 \cdot \frac{\log f}{\log {\log f}}.
    \end{equation}
    And, the statement $G \left( X \right) = o \left( \log {\log X} \right)$ implies that for any $\varepsilon > 0$, there exists a constant $C_{G,\varepsilon}> e^{2e}$ (depending on $G$ and $\varepsilon$) such that 
    \begin{equation*}
        1.38402 \cdot \log 3 \cdot \frac{G \left( X \right)}{\log {\log {\sqrt{X}}}} < \varepsilon
    \end{equation*}
    for all $X > C_{G,\varepsilon}$. Hereinafter, we assume $X > C_{G,\varepsilon}$. 
    
    Since $\alpha \left( f \right) , \beta \left( f \right) , \alpha' \left( f \right), \beta' \left( f \right)$
    $ \ll f^{-2}$, we obtain
    \begin{align*}
        E_1 &\ll X \sum_{f \ge \sqrt{X}} f^{-2} 3^{\left( \psi \left( f \right) +l +1 \right) G \left( X \right)}\\
        &= 3^{\left( l +1 \right) G \left( X \right)} X \sum_{f \ge \sqrt{X}} f^{-2} \exp \left( 1.38402 \cdot \log 3 \cdot \frac{G \left( X \right) \log f}{\log \log f} \right)\\
        &\le 3^{\left( l +1 \right) G \left( X \right)} X \sum_{f \ge \sqrt{X}} f^{-2} \exp \left( 1.38402 \cdot \log 3 \cdot \frac{G \left( X \right) \log f}{\log \log \sqrt{X}} \right)\\
        &\le 3^{\left( l +1 \right) G \left( X \right)} X \sum_{f \ge \sqrt{X}} f^{-2} \exp \left( \varepsilon \log f \right)\\
        &= 3^{\left( l +1 \right) G \left( X \right)} X \sum_{f \ge \sqrt{X}} f^{-2+ \varepsilon}\\
        &\ll 3^{l G \left( X \right)} X^{1/2 + \varepsilon/2}
    \end{align*}
        Since the function $\frac{\log x}{\log \log x}$ is monotonically increasing for $x \ge e^e$, the inequality
    \begin{equation*}
        \psi \left( f \right) \le 1.38402 \cdot \frac{\log \sqrt{X}}{\log {\log \sqrt{X}}}
    \end{equation*}
    holds for $e^e \le f < \sqrt{X}$. For the case $f < e^e$, we obtain the same estimate from the fact that $\psi \left( f \right) \le 2$. 
    Thus, we have
    \begin{align*}
        E_2 &\ll \sum_{\substack{f < \sqrt{X} \\ f \in \SF \left( \Sigma \right) }} 3^{\left( \psi \left( f \right) +l +1 \right) \Re \left( z \right)} 
        \left( E \left( X, f \right) + E' \left( X, f \right) \right)\\
        &\ll 3^{\left( l +1 \right) G \left( X \right)} 
        \exp \left( 1.38402 \cdot \log 3 \cdot \frac{G \left( X \right) \log \sqrt{X}}{\log {\log \sqrt{X}}} \right) \sum_{\substack{f < \sqrt{X} \\ f \in \SF \left( \Sigma \right) }} 
        \left( E \left( X, f \right) + E' \left( X, f \right) \right)\\
        &\ll_\varepsilon 3^{l G \left( X \right)} 
        \exp \left( \frac{\varepsilon}{2} \log X \right) {c \left( \Sigma \right)}^{2/3} X^{2/3 + \varepsilon /2}\\
        &= 3^{l G \left( X \right)} {c \left( \Sigma \right)}^{2/3} X^{2/3 + \varepsilon}.
    \end{align*}
    This completes the proof of Theorem~\ref{main2}. 
\end{proof}

Under suitable convergence assumptions, the identity
\begin{equation*}
    \sum_{f} 3^{\left( \psi \left( f \right) + l \right)z} \prod_{p \mid f} \varphi \left( p \right)
    = 3^{lz} \sum_f \prod_{p \mid f} 3^{\psi \left( p \right) z} \varphi \left( p \right)
    = 3^{lz} \prod_{p \in \mathcal{P}} \left( 1+ 3^{\psi \left( p \right) z} \varphi \left( p \right) \right)
\end{equation*}
holds in general, and in particular we can easily verified that the constants $\xi_{z,\Sigma}$ and $\eta_{z,\Sigma}$ defined in the statement can be expressed as
\begin{align}
    \alpha^{\Sigma} \left( z \right) &\coloneqq
    \begin{aligned}[t]
        \frac{3^{lz} \cdot C_{\infty} \left( \Sigma_{\infty} \right)}{12 \zeta \left( 2 \right)}a_t
        &\left( C_3 \left( \Sigma_3 \right) + \left( 3^z -1 \right) C_3 \left( \Sigma_3 \cap \mathscr{T}_3^{\left( 1 \right)} \right) \right)\\
        &\quad\quad\quad\quad\quad 
        \times \prod_{p \mid c \left( \Sigma \right)} C_p \left( \Sigma_p \right) \left( 1+ a_p \right)
        \prod_{p \in \mathcal{P} \cup \left\{ 3 \right\}}
        \left( 1 + {n_p}^z a_p \right),
    \end{aligned} \label{mainkeisuu21}
    \\
    \beta^{\Sigma} \left( z \right) &\coloneqq 
    \begin{aligned}[t]
        \frac{3^{lz} \cdot 4 \zeta \left( 1/3 \right) K_{\infty} \left( \Sigma_{\infty} \right)}{5 {\Gamma \left( 2/3 \right)}^3} b_t
        &\prod_p \left( 1- \frac{p^{1/3} +1}{p \left( p+1 \right)} \right)
        \left( K_3 \left( \Sigma_3 \right) + \left( 3^z -1 \right) K_3 \left( \Sigma_3 \cap \mathscr{T}_3^{\left( 1 \right)} \right) \right)\\
        &\quad \times \prod_{p \mid c \left( \Sigma \right)} K_p \left( \Sigma_p \right) \left( 1+ b_p \right) 
        \prod_{p \in \mathcal{P} \cup \left\{ 3 \right\}}
        \left( 1 + {n_p}^z b_p \right).
    \end{aligned} \label{mainkeisuu22}
\end{align}

If we set $\Sigma_p = \mathscr{A}_p$ for all primes $p$ and $z=1$, we obtain Theorem~\ref{mtkairyou3}. 

\section*{Acknowledgements}
The author would like to express his sincere gratitude to his supervisor, Prof. Takashi Taniguchi, for his constant guidance and many valuable suggestions. The author is also deeply grateful to Prof. Frank Thorne for his insightful comments and suggestions on an earlier version of this paper. 

\printbibliography

\end{document}